\documentclass[a4paper,10pt,hidelinks]{amsart}
\usepackage{graphicx} 
\usepackage[a4paper, margin=2 cm]{geometry}
\usepackage{hyperref}
\usepackage{amsthm}
\usepackage{thm-restate}
\usepackage{geometry}
\usepackage{amssymb}
\usepackage{amsmath}
\usepackage{cleveref}
\usepackage{comment}
\usepackage{xcolor}
\usepackage{blkarray}
\usepackage{lmodern}
\usepackage[english]{babel}
\usepackage{float}
\usepackage{tikz}
\usetikzlibrary{arrows.meta,arrows}
\usetikzlibrary{arrows,decorations.markings}
\usetikzlibrary{decorations.pathmorphing}
\usepackage[square,sort,comma,numbers]{natbib}

\newtheorem{theorem}{Theorem}[section]

\newtheorem*{claim*}{Claim}
\newtheorem{proposition}[theorem]{Proposition}
\newtheorem{conjecture}[theorem]{Conjecture}

\theoremstyle{definition}

\newtheorem*{definition*}{Definition}

\author
{
Anders Martinsson
}

\thanks{A.M.: Institute of Theoretical Computer Science, Department of Computer Science, ETH Z\"{u}rich, Switzerland.}
\thanks{\texttt{anders.martinsson@inf.ethz.ch}}

\author
{
Raphael Steiner
}
\thanks{R.S.: Institute for Operations Research, Department of Mathematics, ETH Z\"{u}rich, Switzerland. }
\thanks{\texttt{raphaelmario.steiner@math.ethz.ch}.
\thanks{The research of R.S. was funded by the SNSF Ambizione Grant No. 216071 of the Swiss National Science Foundation.}
}

\date{\today}

\title{Local Shearer bound}

\begin{document}
\maketitle

\begin{abstract}
We prove the following local strengthening of Shearer's classic bound on the independence number of triangle-free graphs: For every triangle-free graph $G$ there exists a probability distribution on its independent sets such that every vertex $v$ of $G$ is contained in a random independent set drawn from the distribution with probability $(1-o(1))\frac{\ln d(v)}{d(v)}$. This resolves the main conjecture raised by Kelly and Postle (2018) about fractional coloring with local demands, which in turn confirms a conjecture by Cames van Batenburg et al.~(2018) stating that every $n$-vertex triangle-free graph has fractional chromatic number at most $(\sqrt{2}+o(1))\sqrt{\frac{n}{\ln(n)}}$. Addressing another conjecture posed by Cames van Batenburg et al.,~we also establish an analogous upper bound in terms of the number of edges.

To prove these results we establish a more general technical theorem that works in a weighted setting. As a further application of this more general result, we obtain a new spectral upper bound on the fractional chromatic number of triangle-free graphs: We show that every triangle-free graph $G$ satisfies $\chi_f(G)\le (1+o(1))\frac{\rho(G)}{\ln \rho(G)}$ where $\rho(G)$ denotes the spectral radius. This improves the bound implied by Wilf's classic spectral estimate for the chromatic number by a $\ln \rho(G)$ factor and makes progress towards a conjecture of Harris on fractional coloring of degenerate graphs.
\end{abstract}

\section{Introduction}

One of the most fundamental problems in all of combinatorics concerns bounding the famous Ramsey number $R(\ell,k)$, which may be defined as the smallest number $n$ such that every graph on $n$ vertices contains either a clique of size $\ell$ or an independent set of size $k$. The first highly challenging instance of this general problem is the determination of the Ramsey-numbers $R(3,k)$, posed as a price-money question by Erd\H{o}s already back in 1961~\cite{erdos61} (see also~\href{https://www.erdosproblems.com/165}{this Erd\H{o}s problem page entry}). Currently, the best known asymptotic bounds are
$$\left(\frac{1}{4}-o(1)\right)\frac{k^2}{\ln k}\le R(3,k)\le (1+o(1))\frac{k^2}{\ln k}.$$
The lower bound was established by an analysis of the famous \emph{triangle-free process} independently by Fiz Pontiveros, Griffiths and Morris~\cite{fiz} and Bohman and Keevash~\cite{bohman} in 2013. Proving an upper bound on $R(3,k)$ is equivalent to establishing a lower bound on the \emph{independence number} $\alpha(G)$ (i.e., the size of a largest independent set in $G$) for all triangle-free graphs on $n$ vertices. It turns out that the upper bound on $R(3,k)$ stated above is a straightforward consequence of a much more general such lower bound on the independence number of triangle-free graphs, established in a landmark result by Shearer in 1983~\cite{shearer83}. 

Shearer's classic bound states that every triangle-free graph on $n$ vertices with average degree $\overline{d}$ satisfies $\alpha(G)\ge \frac{(1-\overline{d})+\overline{d}\ln \overline{d}}{(\overline{d}-1)^2}n=(1-o(1))\frac{\ln \overline{d}}{\overline{d}}n$, which improves over the easy lower bound $\alpha(G)\ge \frac{n}{\overline{d}+1}$ which follows from Tur\'{a}n's theorem by roughly a $\ln \overline{d}$-factor. In a refinement of this result, Shearer~\cite{shearer91} proved in 1991 that every triangle-free graph satisfies
$\alpha(G)\ge \sum_{v\in V(G)}g(d_G(v))$, where $d_G(v)$ denotes the degree of $v$ in $G$ and $g(d)=(1-o(1))\frac{\ln d}{d}$ is a recursively defined function. This bound is slightly better than Shearer's first bound in terms of the average degree for graphs with unbalanced degree sequences. These two classic bounds on the independence number of triangle-free graphs due to Shearer have found widespread application across many topics in extremal combinatorics and nowadays form one of its basic tools. By a result of Bollob\'{a}s~\cite{bollobas}, it is known that Shearer's bounds are tight up to a multiplicative factor of $2$. Because of the relation to the Ramsey-numbers $R(3,k)$ discussed above, any constant factor improvement of Shearer's longstanding bounds would be a major breakthrough in Ramsey theory.
Due to this, a lot of research has been devoted to finding strengthenings and generalizations of Shearer's bounds: We refer to~\cite{kang} for a recent survey covering Shearer's bound and related results in the theory of graph coloring and to~\cite{alonsparse,davies20,davies21,davies2018,davies202,dawan,pirot21} for some (mostly recent) extensions and generalizations of Shearer's bounds. 

The study of lower bounds on the independence number, such as the aforementioned bounds of Shearer, is closely connected to the theory of graph coloring. Recall that in a \emph{proper graph coloring} vertices are assigned colors such that neighboring vertices have distinct colors, and the \emph{chromatic number} $\chi(G)$ of a graph $G$ is the smallest amount of colors required to properly color $G$. It is easily seen by the definition (by considering a largest ``color class'')  that every graph $G$ on $n$ vertices has an independent set of size at least $\frac{n}{\chi(G)}$. An even stronger lower bound on the independence number is provided by the well-known \emph{fractional chromatic number} $\chi_f(G)$ of the graph.  The fractional chromatic number has many different equivalent definitions (see the standard textbook~\cite{scheinerman} on fractional coloring as a reference). Here, we shall find the following definition convenient: $\chi_f(G)$ is the minimum real number $r\ge 1$ such that there exists a probability distribution on the independent sets of $G$ such that a random independent set $I$ sampled from this distribution contains any given vertex $v\in V(G)$ with probability at least $\frac{1}{r}$. By considering the expected size of a random set drawn from such a distribution, one immediately verifies that $\alpha(G)\ge \frac{n}{\chi_f(G)}$ holds for every graph $G$. In general, the latter lower bound $\frac{n}{\chi_f(G)}$ on the independence number is stronger than the lower bound $\frac{n}{\chi(G)}$, as there are graphs (namely, the so-called Kneser graphs~\cite{baranyi,lovasz}) for which $\chi_f(G)$ is much smaller than $\chi(G)$.

Given these lower bounds of the independence number in terms of the (fractional) chromatic number, it is natural to ask whether there are analogues or strengthenings of Shearer's bounds that provide corresponding \emph{upper} bounds for the (fractional) chromatic number. A prime example of such a result is a recent breakthrough of Molloy~\cite{molloy}, who proved that $\chi(G)\le (1+o(1))\frac{\Delta}{\ln \Delta}$ for every triangle-free graph $G$ with maximum degree $\Delta$, where the $o(1)$-term vanishes as $\Delta\rightarrow \infty$. This strengthened a previous bound of the form $O\left(\frac{\Delta}{\ln \Delta}\right)$ due to Johansson~\cite{johansson} and recovers Shearer's independence number bound in the case of regular graphs in a stronger form. As with Shearer's bound, it is known that Molloy's bound is optimal up to a factor of $2$, and improving the constant $1$ to any constant below $1$ would be a major advance in the field. Several interesting strengthenings and generalizations of Johansson's and Molloy's results have been proved in the literature, see e.g.~\cite{AKS,andersonbernshteyn,anderson,bernshteynbrazel,bonamy,bradshaw,davies20,hurley,hurleypirot,pirot21} for some selected examples.

\medskip

\paragraph*{\textbf{Our results.}} In this paper, we shall be concerned with the following conjecture from 2018 posed by Kelly and Postle~\cite{KP} that claims a ``local strengthening'' of Shearer's bounds that can also be seen as a degree-sequence generalization of Molloy's bound for fractional coloring. 

\begin{conjecture}[Local fractional Shearer/Molloy, cf.~Conjecture~2.2, Proposition~1.4 in~\cite{KP}]\label{con:localshearer}
For every triangle-free graph there exists a probability distribution on its independent sets such that every vertex $v\in V(G)$ appears with probability at least $(1-o(1))\frac{\ln d_G(v)}{d_G(v)}$ in a random independent set chosen from the distribution. Here, the $o(1)$ term represents any function that tends to $0$ as the degree grows.
\end{conjecture}

To see that this conjecture indeed forms a strengthening of Shearer's bounds, note that the expected size of a random independent set drawn from a distribution as given by the conjecture is 
$$\sum_{v\in V(G)}(1-o(1))\frac{\ln d_G(v)}{d_G(v)},$$
which recovers Shearer's second (stronger) lower bound~\cite{shearer91} on the independence number up to lower-order terms. But on top of that, and this explains the word ``local'' in the name of the conjecture, the distribution in Conjecture~\ref{con:localshearer} guarantees that \emph{every} vertex can be expected to be contained in the random independent set a good fraction of the time (and lower degree vertices are contained proportionally more frequently). This relates back to the previously discussed fractional chromatic number, and, for instance, directly implies that $\chi_f(G)\le (1+o(1))\frac{\Delta(G)}{\ln \Delta(G)}$ for every triangle-free graph, which recovers the fractional version of Molloy's bound. 

Adding to that, Conjecture~\ref{con:localshearer} connects to several other notions of graph coloring discussed in detail by Kelly and Postle, see in particular~\cite[Proposition 1.4]{KP} which provides many different equivalent formulations of Conjecture~\ref{con:localshearer}. One of these involves the notion of \emph{fractional coloring with local demands} introduced by Dvo\v{r}\'{a}k, Sereni and Volec~\cite{dvorak}. Following Kelly and Postle~\cite{KP}, given a graph $G$ and a so-called \emph{demand function} $h:V(G)\rightarrow [0,1]$ that assigns to each vertex its individual ``demand'', and \emph{$h$-coloring} of a graph $G$ is a mapping $c:V(G)\rightarrow 2^{[0,1]}$ that assigns to every vertex $v\in V(G)$ a measurable subset $c(v)\subseteq [0,1]$ of measure at least $h(v)$, in such a way that adjacent vertices in $G$ are assigned disjoint subsets. Since the function $h$ does not have to be constant but can depend on local information concerning the vertex $v$ in $G$, this setting extends the usual paradigm of graph coloring in a local manner. Kelly and Postle~\cite[Proposition~1.4]{KP} proved that Conjecture~\ref{con:localshearer} is equivalent to the statement that every triangle-free graph has an $h$-coloring, where $h:V(G)\rightarrow [0,1]$ is a function depending only on the vertex-degrees such that $h(v)=(1-o(1))\frac{\ln d_G(v)}{d_G(v)}$. We refer to the extensive introduction of~\cite{KP} for further applications of the conjecture. 

In one of their main results, Kelly and Postle~\cite[Theorem 2.3]{KP} proved a relaxation of Conjecture~\ref{con:localshearer}, replacing the bound $(1-o(1))\frac{\ln d_G(v)}{d_G(v)}$ with the asymptotically weaker $\left(\frac{1}{2e}-o(1)\right)\frac{\ln d_G(v)}{d_G(v) \ln \ln d_G(v) }$.
As the first main result of this paper, we fully resolve Conjecture~\ref{con:localshearer}.

\begin{theorem}\label{thm:main2}
For every triangle-free graph $G$ there exists a probability distribution $\mathcal{D}$ on the independent sets of $G$ such that $$\mathbb{P}_{I\sim D}[v\in I]\ge (1-o(1))\frac{\ln(d_G(v))}{d_G(v)}$$ for every $v\in V(G)$. Here the $o(1)$ represents a function of $d_G(v)$ that tends to $0$ as the degree grows.
\end{theorem}

A pleasing consequence of Theorem~\ref{thm:main2} is that it can also be used to fully resolve another conjecture about fractional coloring raised in 2018 by Cames van Batenburg, de Joannis de Verclos, Kang and Pirot~\cite{cames}: Motivated by the aforementioned problem of estimating the Ramsey-number $R(3,k)$, in 1967 Erd\H{o}s asked the fundamental question of determining the maximum chromatic number of triangle-free graphs. An observation of Erd\H{o}s and Hajnal~\cite{erdoshajnal} combined with Shearer's bound implies an upper bound $(2\sqrt{2}+o(1)))\sqrt{\frac{n}{\ln n}}$ for this problem. In recent work of Davies and Illingworth~\cite{freddie}, this upper bound was improved by a $\sqrt{2}$-factor to the current state of the art $(2+o(1))\sqrt{\frac
{n}{\ln n}}$. The current best lower bound for this quantity is $(1/\sqrt{2}-o(1))\sqrt{\frac{n}{\ln n}}$, coming from the aforementioned lower bounds on $R(3,k)$~\cite{fiz,bohman}. 

Cames van Batenburg et al.~\cite{cames} recently studied the natural analogue of this question for fractional coloring and made the following conjecture. 
\begin{conjecture}[cf.~Conjecture~4.3 in~\cite{cames}]\label{con:fractionalcoloring}
As $n\rightarrow \infty$, every triangle-free graph on $n$ vertices has fractional chromatic number at most $(\sqrt{2}+o(1))\sqrt{\frac{n}{\ln n}}$.  
\end{conjecture}
In one of their main results~\cite[Theorem~1.4]{cames}, Cames van Batenburg et al. proved the fractional version of the result of Davies an Illingworth, namely an upper bound $(2+o(1))\sqrt{\frac{n}{\ln n}}$ on the fractional chromatic number.
Using a connection between Conjectures~\ref{con:localshearer} and~\ref{con:fractionalcoloring} proved by Kelly and Postle~\cite[Proposition~5.2]{KP}, we are able to confirm Conjecture~\ref{con:fractionalcoloring} too. 
\begin{theorem}\label{thm:main3}
The maximum fractional chromatic number among all $n$-vertex triangle-free graphs is at most $$(\sqrt{2}+o(1))\sqrt{\frac{n}{\ln(n)}}.$$ 
\end{theorem}

We also prove a similar upper bound on the fractional chromatic number of triangle-free graphs in terms of the number of edges, as follows.

\begin{theorem}\label{thm:edgebound}
The maximum fractional chromatic number among triangle-free graphs with $m$ edges is at most
 $$(18^{1/3}+o(1)) \frac{m^{1/3}}{(\ln m)^{2/3}}.$$   
\end{theorem}

Theorem~\ref{thm:edgebound} comes very close to confirming another conjecture of Cames van Batenburg et al.~\cite[Conjecture~4.4]{cames}, stating that every triangle-free graph with $m$ edges has fractional chromatic number at most $(16^{1/3}+o(1))m^{1/3}/( \ln m)^{2/3}.$ In fact, we suspect that the constant $16^{1/3}$ in this conjecture may be due to a miscalculation on their end. In particular, the authors of~\cite{cames} claim that their conjectured bound on the fractional chromatic number can be verified in the special case of $d$--regular triangle--free graphs using the upper bound $\chi_f(G)\leq \min\left((1+o(1))d/\ln d, n/d  \right).$ However, assuming $n=(1+o(1))d^2/\ln d$ and thus $m=(1/2+o(1))d^3/\ln d$, this upper bound simplifies to $(1+o(1))d/\ln d = (1+o(1))(2m)^{1/3}/(\ln m^{1/3})^{2/3}=(1+o(1))(18m)^{1/3}/(\ln m)^{2/3}$.

\medskip

To prove our main result, Theorem~\ref{thm:main2}, in Section~\ref{sec:proof} we establish a key technical result that goes beyond Theorem~\ref{thm:main2} and generalizes it to a vertex-weighted setting (Theorem~\ref{thm:aux}). As a consequence of this more general result, we obtain a new spectral upper bound on the fractional chromatic number of triangle-free graphs, stated below. In the following, as is standard, $\rho(G)$ denotes the \emph{spectral radius} of $G$, i.e., the spectral radius of its adjacency matrix. 
\begin{theorem}\label{thm:main4}
Every triangle-free graph $G$ satisfies $$\chi_f(G)\le (1+o(1))\frac{\rho(G)}{\ln \rho(G)},$$ where the $o(1)$ term represents a function of $\rho(G)$ that tends to $0$ as $\rho(G)$ grows.
\end{theorem}
Theorem~\ref{thm:main4} lines up nicely with a rich area of research that is concerned with spectral bounds on the (fractional) chromatic number, see e.g. Chapter~6 of the textbook on spectral graph theory~\cite{chung} by Chung and~\cite{bilu,cvetkovic,guo,hoffman,yuval,mohar,nikiforov} for some small selection of articles on the topic. On the one hand, Theorem~\ref{thm:main4} relates to Wilf's classic spectral bound~\cite{wilf} on the chromatic number, which states that every connected graph $G$ satisfies $\chi(G)\le \rho(G)+1$ with equality if and only if $G$ is an odd cycle or a complete graph. Theorem~\ref{thm:main4} shows that at least for the fractional chromatic number, a $\ln \rho(G)$-factor can be saved compared to Wilf's bound when the graph is assumed to be triangle-free. On the other hand, it is easy to see that $\rho(G)\le \Delta(G)$ for every graph $G$ and in fact, Hayes~\cite{hayes} proved that moreover, every $d$-degenerate graph $G$ satisfies $\rho(G)\le 2\sqrt{d(\Delta(G)-d)}$. This implies that $\rho(G)$ is asymptotically smaller than the maximum degree $\Delta(G)$ for many graphs, which means that Theorem~\ref{thm:main4} may be regarded as a substantial strengthening of the upper bound $\chi_f(G)\le (1+o(1))\frac{\Delta(G)}{\ln \Delta(G)}$ implied by Molloy's result~\cite{molloy}. 

Theorem~\ref{thm:main4} also relates to a conjecture of Harris~\cite{harris}, stating that every triangle-free $d$-degenerate graph $G$ satisfies $\chi_f(G)\le O(\frac{d}{\ln d})$. A proof of this conjecture would have significant ramifications for other problems in graph theory, such as the resolution of a conjecture by Esperet, Kang and Thomassé~\cite{Esperet} on bipartite induced subgraphs, (see~\cite{cames,benny} for a discussion). Since the spectral radius $\rho(G)$ is sandwiched between the degeneracy and the maximum degree of the graph, Theorem~\ref{thm:main4} can be seen as a first step of moving from Molloy's $(1+o(1))\frac{\Delta(G)}{\ln \Delta(G)}$ upper bound to the smaller bound conjectured by Harris. It would be interesting to strengthen Theorem~\ref{thm:main4} to a bound for the chromatic number, as follows.
\begin{conjecture}
    Every triangle-free graph $G$ satisfies
    $\chi(G)\le(1+o(1))\frac{\rho(G)}{\ln \rho(G)}$.
\end{conjecture}
\paragraph*{\textbf{Organization.}} The rest of the paper is structured as follows: In the Section~\ref{sec:proof} we establish our main technical result, Theorem~\ref{thm:aux}. In the following  Section~\ref{sec:deduc} we then derive our main results from Theorem~\ref{thm:aux}.
\section{Key technical result}\label{sec:proof}
In this section, we present the proof of our key technical result, Theorem~\ref{thm:aux} below, which generalizes Theorem~\ref{thm:main2} to a vertex-weighted setting. 

In the following, we denote by\footnote{We remark that the function $f$ is the same function that was used by Shearer in his first paper~\cite{shearer83}.} $f:[0,\infty)\rightarrow \mathbb{R}_+$ the unique continuous extension of $x\rightarrow \frac{(1-x)+x\ln(x)}{(x-1)^2}$ from $[0,\infty)\setminus \{0,1\}$ to $[0,\infty)$. It is not hard to check that $f$ exists and has the following properties:
\begin{itemize}
    \item $f(0)=1, f(1)=\frac{1}{2}$,
    \item $f$ is convex,
    \item $f$ is strictly monotonically decreasing,
    \item $f$ is continuously differentiable on $(0,\infty)$ and satisfies the following differential equation:
    $$x(x-1)f'(x)+(x+1)f(x)=1$$ for every $x>0$, 
    \item $|xf'(x)|<1$ for every $x>0$, and
    \item $f(x)=(1-o(1))\frac{\ln(x)}{x}$ as $x\rightarrow \infty$.
\end{itemize}
In the following, given a weight function $w:V(G)\rightarrow\mathbb{R}_+$ on the vertices of a graph $G$ and a subset $X\subseteq V(G)$, $w(X):=\sum_{v\in X}{w(v)}$ denotes the total weight of $X$.

In this section our goal shall be to establish the following statement, which represents the main technical contribution of the paper at hand and from which all our main results (Theorems~\ref{thm:main2},~\ref{thm:main3},~\ref{thm:edgebound},~\ref{thm:main4}) can be deduced without much work (this will be done in the following section). We believe that the more general result offered by Theorem~\ref{thm:aux} may be of independent interest.

\begin{theorem}\label{thm:aux}
For every triangle-free graph $G$ and every strictly positive weight function $w:V(G)\rightarrow \mathbb{R}_+$ on the vertices there exists a probability distribution $\mathcal{D}$ on the independent sets of $G$ such that $$\mathbb{P}_{I\sim \mathcal{D}}[v\in I]\ge f\left(\frac{w(N_G(v))}{w(v)}\right)$$ for every vertex $v\in V(G)$.
\end{theorem}
\begin{proof}
We prove the statement by induction on $|V(G)|$. In the base case $|V(G)|=1$, there is a unique vertex $v$ of $G$, so $w(N_G(v))=w(\emptyset)=0$ and hence our target probability of the appearance of $v$ in a randomly drawn independent set is $f(0)=1$. This is easily achieved by letting $\mathcal{D}$ be the probability distribution that always picks $\{v\}$, establishing the induction base.

For the induction step, let us assume that $G$ is a triangle-free graph on at least two vertices and that we have already proven the claim of the theorem for all triangle-free graphs with strictly less vertices than $G$. 

Let $K \subseteq [0,1]$ be the set of all $\delta\in [0,1]$ such that for every strictly positive weight function $w:V(G)\rightarrow \mathbb{R}_+$ there exists a probability distribution $\mathcal{D}$ on the independent sets of $G$ such that $\mathbb{P}_{I\sim D}[v\in I]\ge f\left(\frac{w(N_G(v))}{w(v)}\right)-\delta$ for every $v\in V(G)$. Since $f$ takes values in $[0,1]$, we trivially have $1\in K$. Furthermore, we claim that the set $K$ is closed (and thus compact). 
To see this, note that $K=\bigcap_{w:V(G)\rightarrow \mathbb{R}_+} K_w$, where $K_w$ is the set of all $\delta\in [0,1]$ for which there exists a probability distribution $\mathcal{D}$ on the independent sets of $G$ satisfying $\mathbb{P}_{I\sim \mathcal{D}}[v\in I]\ge f\left(\frac{w(N_G(v))}{w(v)}\right)-\delta$ for every $v\in V(G)$. Since intersections of closed sets are closed, it suffices to show that $K_w$ is closed for every fixed $w:V(G)\rightarrow \mathbb{R}_+$. Now, consider the following linear program ($\mathcal{I}(G)$ denotes the collection of all independent sets in $G$):

\begin{align*}
    \text{min}~~ &y\\
    \text{s.t.}~~ y+\sum_{I\in \mathcal{I}(G): v \in I}{x_I}&\ge f\left(\frac{w(N_G(v))}{w(v)}\right)~~(\forall v \in V(G)),
    \\
    \sum_{I\in \mathcal{I}(G)} x_I&=1
    \\
     x_I &\ge 0~~(\forall I \in \mathcal{I}(G)).
\end{align*}
It can easily be checked that this linear program is bounded and feasible, and hence has a unique optimum $y^\ast$. Further, since the constraints of the program encode a probability distribution $\mathcal{D}$ on independent sets with $\mathbb{P}_{I\sim \mathcal{D}}[v\in I]\ge f\left(\frac{w(N_G(v))}{w(v)}\right)-y$, we can see that $K_w=[y^\ast,1]$ is indeed a closed set as desired.

This shows that $K$ is indeed compact and hence has a unique minimum $\delta_0 \in K$. Our goal is to show that $0\in K$ (equivalently, $\delta_0=0$), since this clearly establishes the induction hypothesis for $G$. So, towards a contradiction, let us assume $\delta_0>0$ in the following.

Let us define $\delta:=\delta_0-\frac{\delta_0^2}{8}$. Then, since $\delta\in (0,\delta_0)$ and hence $\delta \notin K$, there exists a positive weight function $w:V(G)\rightarrow \mathbb{R}_+$ such that there exists no probability distribution on the independent sets of $G$ for which every vertex $v$ is contained in an independent set drawn from the distribution with probability at least $f\left(\frac{w(N_G(v))}{w(v)}\right)-\delta$. Since the latter formula is scale-invariant, we may assume w.l.o.g throughout the rest of the proof that $w(V(G))=1$.

Let us pick and fix some $\varepsilon \in (0,1)$ (for now arbitrarily, later on we will assign a concrete value). Let $w':V(G)\rightarrow \mathbb{R}^+$ be a modified vertex-weighting of $G$, defined as $w'(v):=w(v)\cdot \exp\left(\varepsilon w(N_G(v))\right)$ for every $v\in V(G)$. 

Since $\delta_0\in K$, there must exist a probability distribution $\mathcal{D}$ on the independent sets of $G$ such that
$$\mathbb{P}_{I\sim \mathcal{D}}[v\in I]\ge f\left(\frac{w'(N_G(v))}{w'(v)}\right)-\delta_0$$ for every $v\in V(G)$.

For a vertex $u\in V(G)$, let us denote by $\overline{N}_G(u):=\{u\}\cup N_G(u)$ the \emph{closed neighborhood} of $u$ and by $G_u:=G-\overline{N}_G(u)$ the graph obtained from $G$ by deleting this closed neighborhood. By the inductive assumption, for every $u\in V(G)$ there exists a probability distribution $\mathcal{D}_u$ on $G_u$ such that $\mathbb{P}_{I\sim \mathcal{D}_u}[v \in I]\ge f\left(\frac{w'(N_{G_u}(v))}{w'(v)}\right)$ for every $v\in V(G_u)$. 

Let us now define $\varepsilon:=\frac{\delta_0}{4} \in (0,1)$, and let us consider the following process to generate a random independent set $I$ of $G$:
\begin{itemize}
    \item With probability $1-\varepsilon$ (we call this event $A$), draw $I$ randomly from the distribution $\mathcal{D}$ and return $I$.
    \item With probability $\varepsilon$ (we call this event $B:=A^\mathsf{c}$), do the following: First, sample randomly a vertex $u\in V(G)$ where $u$ equals any given vertex $x$ with probability exactly $w(x)$. Then, randomly draw an independent set $I_u$ from the distribution $\mathcal{D}_u$ and return the independent set $I:=\{u\}\cup I_u$. 
\end{itemize}
In the following, let $\mathcal{D}'$ denote the probability distribution on independent sets of $G$ that is induced by the random independent set $I$ created according to the above process. By our choice of the weight function $w$, there must exist some vertex $v\in V(G)$ such that $$\mathbb{P}_{I\sim \mathcal{D}'}[v\in I]<f\left(\frac{w(N_G(v))}{w(v)}\right)-\delta.$$ Our intermediate goal is to give a lower bound on $\mathbb{P}_{I\sim \mathcal{D}'}[v\in I]$. 

To estimate this probability, we stick with the random process described above. We then have
\begin{align*}
\mathbb{P}_{I\sim \mathcal{D}'}[v\in I]
&=(1-\varepsilon)\mathbb{P}_{I\sim \mathcal{D}'}[v\in I|A]+\varepsilon \mathbb{P}_{I\sim \mathcal{D}'}[v\in I|B]\\
&=(1-\varepsilon)\mathbb{P}_{I\sim \mathcal{D}}[v\in I]+\varepsilon\sum_{x\in V(G)}\mathbb{P}_{I\sim \mathcal{D}'}[v\in I|B \wedge \{u=x\}]w(x)\\
&=(1-\varepsilon)\mathbb{P}_{I\sim \mathcal{D}}[v\in I]+\varepsilon w(v)+\varepsilon\sum_{x\in V(G)\setminus \overline{N}_G(v)}\mathbb{P}_{I\sim \mathcal{D}_x}[v\in I]w(x).\end{align*}
By our choice of the distributions $\mathcal{D}$ and $\mathcal{D}_x$, $x\in V(G)$, we have
\begin{align*}
&(1-\varepsilon)\mathbb{P}_{I\sim \mathcal{D}}[v\in I]+\varepsilon\sum_{x\in V(G)\setminus \overline{N}_G(v)}\mathbb{P}_{I\sim \mathcal{D}_x}[v\in I]w(x)\\ &\ge (1-\varepsilon)\left(f\left(\frac{w'(N_G(v))}{w'(v)}\right)-\delta_0\right)+\varepsilon \sum_{x\in V(G)\setminus \overline{N}_G(v)}f\left(\frac{w'(N_{G_x}(v))}{w'(v)}\right)w(x)\\ &=-(1-\varepsilon)\delta_0+(1-\varepsilon)f\left(\frac{w'(N_G(v))}{w'(v)}\right)+\sum_{x\in V(G)\setminus \overline{N}_G(v)}\varepsilon w(x)f\left(\frac{w'(N_{G_x}(v))}{w'(v)}\right).\end{align*}
Since $(1-\varepsilon)+\sum_{x\in V(G)\setminus \overline{N}_G(v)}\varepsilon w(x)=(1-\varepsilon)+\varepsilon\left(1-w(\overline{N}_G(v))\right)=1-\varepsilon w(\overline{N}_G(v))$, the convexity of $f$ implies that
$$(1-\varepsilon)f\left(\frac{w'(N_G(v))}{w'(v)}\right)+\sum_{x\in V(G)\setminus \overline{N}_G(v)}\varepsilon w(x)f\left(\frac{w'(N_{G_x}(v))}{w'(v)}\right)\ge $$
$$\left(1-\varepsilon w(\overline{N}_G(v))\right)f\left(\frac{(1-\varepsilon)w'(N_G(v))+\sum_{x\in V(G)\setminus \overline{N}_G(v)}\varepsilon w(x)w'(N_{G_x}(v))}{w'(v)\left(1-\varepsilon w(\overline{N}_G(v))\right)}\right).$$

The next claim gives a simple upper bound for the expression in the argument of $f$ above.

\medskip

\paragraph*{\textbf{Claim.}} We have that $$\frac{(1-\varepsilon)w'(N_G(v))+\sum_{x\in V(G)\setminus \overline{N}_G(v)}\varepsilon w(x)w'(N_{G_x}(v))}{w'(v)\left(1-\varepsilon w(\overline{N}_G(v))\right)} \le \frac{w(N_G(v))}{w(v)}\cdot e^{\varepsilon (w(v)-w(N_G(v)))}.$$
\begin{proof}[Proof of the Claim.]
We have
\begin{align*}
    &(1-\varepsilon) w'(N_G(v))+\sum_{x\in V(G)\setminus \overline{N}_G(v)}\varepsilon w(x)w'(N_{G_x}(v))\\
    &=\sum_{y\in N_G(v)}{(1-\varepsilon)w'(y)}+\sum_{x\in V(G)\setminus \overline{N}_G(v)}\varepsilon w(x)\sum_{y\in N_G(v)\setminus \overline{N}_G(x)}{w'(y)}\\
    &=\sum_{y\in N_G(v)}\left((1-\varepsilon)+\sum_{x\in V(G)\setminus (\overline{N}_G(v)\cup \overline{N}_G(y))}\varepsilon w(x)\right)w'(y)\\
    &=\sum_{y\in N_G(v)}\left(1-\varepsilon+\varepsilon(1-w(\overline{N}_G(v)\cup \overline{N}_G(y)))\right)w'(y)\\
    &=\sum_{y\in N_G(v)}\left(1-\varepsilon w(\overline{N}_G(v)\cup \overline{N}_G(y))\right)w'(y).
\end{align*}
Note that for every $y\in N_G(v)$, we have $\overline{N}_G(v)\cup \overline{N}_G(y)=N_G(v)\cup N_G(y)$. Furthermore, since $G$ is triangle-free, the sets $N_G(v)$ and $N_G(y)$ are disjoint, and thus we have $w(\overline{N}_G(v)\cup \overline{N}_G(y))=w(N_G(v))+w(N_G(y))$. This implies
\begin{align*}
&\frac{(1-\varepsilon)w'(N_G(v))+\sum_{x\in V(G)\setminus \overline{N}_G(v)}\varepsilon w(x) w'(N_{G_x}(v))}{w'(v)\left(1-\varepsilon w(\overline{N}_G(v))\right)}\\
&=\frac{
1}{w'(v)}\sum_{y\in N_G(v)}{\frac{1-\varepsilon w(N_G(v))-\varepsilon w(N_G(y))}{1-\varepsilon w(\overline{N}_G(v))}w'(y)}\\
&=\frac{
1}{w'(v)}\sum_{y\in N_G(v)}{\left(1-\varepsilon \frac{w(N_G(y))-w(v)}{1-\varepsilon w(\overline{N}_G(v))}\right)w'(y)}\\
&\le \frac{
1}{w'(v)}\sum_{y\in N_G(v)}{\left(1-\varepsilon (w(N_G(y))-w(v))\right)w'(y)}\\
&\le \frac{1}{w'(v)}\sum_{y\in N_G(v)}\exp\left(-\varepsilon (w(N_G(y))-w(v))\right)\cdot w(y)\exp\left(\varepsilon w(N_G(y))\right)\\
&=\frac{1}{w'(v)}\exp\left(\varepsilon w(v)\right)w(N_G(v))\\
&=\frac{w(N_G(v))}{w(v)}\cdot \exp\left(\varepsilon(w(v)-w(N_G(v)))\right),
\end{align*}
where we used the definition of $w'$ in the last and third to last line. This concludes the proof of the claim.
\end{proof}
Using the claim and the previously established inequalities (using that $f$ is monotonically decreasing), it follows that $\mathbb{P}_{I\sim \mathcal{D}'}[v\in I]$ is lower-bounded by
$$\varepsilon w(v)-(1-\varepsilon)\delta_0+\left(1-\varepsilon w(\overline{N}_G(v))\right)f\left(\frac{w(N_G(v))}{w(v)}\cdot \exp\left(\varepsilon (w(v)-w(N_G(v)))\right)\right).$$

Let us now go about estimating the above expression. By Taylor expansion, it is not hard to verify that the inequality $\exp(z)\le 1+z+2z^2$ holds for every $z\in [-1,1]$. Let us set $x:=\frac{w(N_G(v))}{w(v)}$, $z:=\varepsilon(w(v)-w(N_G(v)))$ and $y:=x\cdot \exp(z)$. Note that since $f$ is convex and differentiable, we have the inequality $f(y)\ge f(x)+f'(x)(y-x)$. Since $w(v), w(N_G(v))\le w(V(G))=1$, we obtain $|z|\le \varepsilon<1$ and thus $y\le x(1+z+2z^2)$. Since $f$ is monotonically decreasing, we have $f'(x)<0$. Putting these facts together, it follows that
\begin{align*}&f\left(\frac{w(N_G(v))}{w(v)}\cdot \exp\left(\varepsilon (w(v)-w(N_G(v)))\right)\right)=f(y)\\
&\ge f(x)+f'(x)(y-x)\\
&\ge f(x)+f'(x)\cdot x\cdot (z+2z^2)\\
& \ge f(x)+xzf'(x)-2\varepsilon^2,\end{align*}
where we used that $|x\cdot f'(x)|\le 1$ for every $x>0$ and that $|z|\le \varepsilon$ in the last line. Plugging this estimate into the above lower bound for $\mathbb{P}_{I\sim \mathcal{D}'}[v\in I]$ and using that by our choice of $v$, we have $\mathbb{P}_{I\sim\mathcal{D}'}[v\in I]<f\left(\frac{w(N_G(v))}{w(v)}\right)-\delta$, we find:

\begin{align*}
    &f(x)-\delta>\mathbb{P}_{I\sim \mathcal{D}'}[v\in I]\\
    &\ge -(1-\varepsilon)\delta_0+\varepsilon w(v)+(1-\varepsilon (w(v)+w(N_G(v))))\cdot (f(x)+\underbrace{xzf'(x)-2\varepsilon^2}_{<0})\\
    &>-(1-\varepsilon)\delta_0+\varepsilon w(v)+f(x)+xzf'(x)-2\varepsilon^2-\varepsilon(w(v)+w(N_G(v)))f(x).
\end{align*}
Rearranging yields 
$$\delta_0-\delta>\varepsilon\delta_0-2\varepsilon^2+\varepsilon w(v)+xzf'(x)-\varepsilon (w(v)+w(N_G(v)))f(x).$$ Using that $x=\frac{w(N_G(v))}{w(v)}$ and $z=\varepsilon w(v) (1-x)$, we can simplify as follows.

\begin{align*}
    &\varepsilon w(v)+xzf'(x)-\varepsilon(w(v)+w(N_G(v)))f(x)\\
    &=\varepsilon w(v)\left(1+x(1-x)f'(x)-(1+x)f(x)\right)=0,
\end{align*}
where we used the differential equation satisfied by $f$ in the last step. Hence, we have proven the inequality $\delta_0-\delta>\varepsilon \delta_0-2\varepsilon^2$. Recalling our definitions $\delta:=\delta_0-\frac{\delta_0^2}{8}$ and $\varepsilon:=\frac{\delta_0}{4}$ we can now see that the above inequality implies $\frac{\delta_0^2}{8}>\frac{\delta_0^2}{8}$, which is absurd. This is the desired contradiction, which shows that our initial assumption, namely that $\delta_0>0$, was wrong. Hence, we have shown that $\delta_0=0$, establishing the inductive claim for $G$. This concludes the proof of the theorem by induction.
\end{proof}

\section{Proofs of Theorems~\ref{thm:main2},~\ref{thm:main3},~\ref{thm:edgebound},~\ref{thm:main4}}\label{sec:deduc}

In this section we use Theorem~\ref{thm:aux} established in the previous sections to deduce our three main results. 

Let us start with Theorem~\ref{thm:main2}, which is a simple corollary of Theorem~\ref{thm:aux} by simply using the all-$1$ weight assignment.

\begin{proof}[Proof of Theorem~\ref{thm:main2}]
Let $G$ be any given triangle-free graph, and let $w:V(G)\rightarrow \mathbb{R}^+$ be defined as $w(v):=1$ for every $v\in V(G)$. Then $\frac{w(N_G(v))}{w(v)}=d_G(v)$ for every vertex $v\in V(G)$, and hence by Theorem~\ref{thm:aux} there exists a probability distribution $\mathcal{D}$ on independent sets of $G$ such that 
$$\mathbb{P}_{I\sim \mathcal{D}}[v\in I]\ge f(d_G(v))$$ for every $v\in V(G)$. Since $f(x)=(1-o(1))\frac{\ln(x)}{x}$, this establishes Theorem~\ref{thm:main2}.
\end{proof}

Next, let us deduce Theorem~\ref{thm:main3}. This, in fact, can be derived from Theorem~\ref{thm:main2} using the following relationship between Conjectures~\ref{con:localshearer} and~\ref{con:fractionalcoloring} proved by Kelly and Postle~\cite[Proposition~5.2]{KP}:

\begin{proposition}\label{prop:KP}
For every $\varepsilon, c>0$, the following holds for sufficiently large $n$. Let $G$ be a triangle-free graph on $n$ vertices with demand function $h$ such that $h(v)\ge c\frac{\ln d_G(v)}{d_G(v)}$ for every $v\in V(G)$. If $G$ has an $h$-coloring, then $$\chi_f(G)\le (\sqrt{2/c}+\varepsilon)\sqrt{\frac{n}{\ln n}}.$$
\end{proposition}

With this statement at hand, we can now easily deduce Theorem~\ref{thm:main3}. 

\begin{proof}[Proof of Theorem~\ref{thm:main3}]
The statement of Theorem~\ref{thm:main3} is equivalent to showing that for every fixed $\delta>0$ and $n$ sufficiently large in terms of $\delta$, every triangle-free graph $G$ on $n$ vertices satisfies $\chi_f(G)\le (\sqrt{2}+\delta)\sqrt{\frac{n}{\ln n}}$. Let $\varepsilon>0$ and $0<c<1$ (only depending on $\delta$) be chosen such that $\sqrt{2/c}+\varepsilon<\sqrt{2}+\delta$. By Proposition~\ref{prop:KP} there exists $n_0=n_0(\varepsilon,c)\in \mathbb{N}$ such that every triangle-free graph $G$ with $n\ge n_0$ vertices that admits an $h$-coloring for some demand function $h$ satisfying $h(v)\ge c\frac{\ln d_G(v)}{d_G(v)}$ for all $v\in V(G)$, has fractional chromatic number at most $(\sqrt{2/c}+\varepsilon)\sqrt{\frac{n}{\ln n}}\le (\sqrt{2}+\delta)\sqrt{\frac{n}{\ln n}}$. By~\cite[Proposition~1.4~(c)]{KP} the latter statement is equivalent to the following:
Every triangle-free graph on $n\ge n_0$ vertices that admits a probability distribution on its independent sets such that each vertex $v$ is included with probability at least $c\frac{\ln d_G(v)}{d_G(v)}$ in a randomly drawn independent set has fractional chromatic number at most $(\sqrt{2}+\delta)\sqrt{\frac{n}{\ln n}}$. 

Since $c<1$, Theorem~\ref{thm:main2} implies that there exists a constant $D=D(c)$ such that every triangle-free graph of minimum degree at least $D$ admits a probability distribution on its independent sets where each vertex $v$ is included in a randomly drawn independent set with probability at least $c\frac{\ln d_G(v)}{d_G(v)}$. Putting this together with the statement above, we immediately obtain that every triangle-free graph on $n\ge n_0$ vertices with minimum degree at least $D$ has fractional chromatic number at most $(\sqrt{2}+\delta)\sqrt{\frac{n}{\ln n}}$. 

Let $n_1$ be an integer chosen large enough such that $(\sqrt{2}+\delta)\sqrt{\frac{n_1}{\ln n_1}}>\max\{D+1,n_0\}$. We now claim that \emph{every} triangle-free graph on $n\ge n_1$ vertices has fractional chromatic number at most $(\sqrt{2}+\delta)\sqrt{\frac{n}{\ln n}}$, which is the statement that we wanted to prove initially. Let $G$ be any given triangle-free graph on $n\ge n_1$ vertices. Let $G'$ be the subgraph of $G$ obtained by repeatedly removing vertices of degree less than $D$ from $G$, until no such vertices are left ($G'$ is the so-called the \emph{$D$-core} of $G$). Then $G'$ is a triangle-free graph that is either empty or has minimum degree at least $D$. Hence, we either have $|V(G')|<n_0$ and thus $\chi_f(G')<n_0$, or $|V(G')|\ge n_0$ and thus $\chi_f(G')\le (\sqrt{2}+\delta)\sqrt{\frac{|V(G')|}{\ln |V(G')|}}\le (\sqrt{2}+\delta)\sqrt{\frac{n}{\ln n}}$.

Pause to verify that $\chi_f(G)\le \max\{\chi_f(G-v),d_G(v)+1\}$ holds for every graph $G$ and every vertex $v\in V(G)$. Repeated application of this fact combined with the definition of $G'$ now implies that $$\chi_f(G)\le \max\{\chi_f(G'),D+1\}\le \max\left\{n_0,(\sqrt{2}+\delta)\sqrt{\frac{n}{\ln n}},D+1\right\}=(\sqrt{2}+\delta)\sqrt{\frac{n}{\ln n}},$$

as desired. Here, we used our choice of $n_1$ and that $n\ge n_1$ in the last step. This concludes the proof.
\end{proof}

Next, let us prove the upper bound on the fractional chromatic number of triangle-free graphs with a given number of edges stated in Theorem~\ref{thm:edgebound}. Interestingly, it can be deduced by applying Theorem~\ref{thm:aux} with two different vertex-weight functions following a similar proof idea to Proposition \ref{prop:KP}.
\begin{proof}[Proof of Theorem~\ref{thm:edgebound}]
Let $G$ be any given triangle-free graph with $m$ edges. To prove the upper bound on the fractional chromatic number, w.l.o.g. it suffices to consider the case when $G$ has no isolated vertices. By definition of the fractional chromatic number, we have to show that there exists a probability distribution on the independent sets of $G$ for which a randomly drawn independent set contains any given vertex of $G$ with probability at least $ (1-o(1))(\ln m)^{2/3}/(18 m)^{1/3}$. To construct such a distribution, we consider the following process to generate a random independent set $I$ in $G$. With probability $1/3$ pick $I$ as in Theorem~\ref{thm:aux} with the weight function defined as $w_1(v):=1$ for every $v\in V(G)$, with probability $1/3$ we pick $I$ as in Theorem~\ref{thm:aux} using the weight function $w_2(v):=d_G(v)$ for every $v\in V(G)$, and with probability $1/3$ we pick a random vertex $u$ in $G$ with $\mathbb{P}[u=v]=d_G(v)/2m$ for every $v\in V(G)$ and let $I$ be its neighborhood (which is clearly an independent set in $G$).

It follows that
$$\mathbb{P}[v\in I]\geq \frac{1}{3} f(d_G(v))+\frac{1}{3} f(S_G(v)/d_G(v)) + \frac{1}{3} \frac{S_G(v)}{2m},$$
for every $v\in V(G)$, where $S_G(v)$ denotes the sum of degrees over all neighbors of $v$ in $G$. It suffices to show that the right-hand side is at least $(1-o(1)) (\ln m)^{2/3}/(18 m)^{1/3}$ for all vertices $v$.

Observe that if either $d_G(v)< m^{1/3}$ or $S_G(v)/d_G(v) < m^{1/3}$, then the desired inequality is already satisfied with room to spare from the first and second terms respectively. Otherwise, if $d_G(v)\geq m^{1/3}$ and $S_G(v)/d(v) \geq m^{1/3}$, we have
\begin{align*}
&\frac{1}{3} f(d_G(v))+\frac{1}{3} f(S_G(v)/d_G(v)) + \frac{1}{3} \frac{S_G(v)}{2m}\\
&\qquad = \frac{1}{3} \frac{(1-o(1))\ln d_G(v)}{d_G(v)} +\frac{1}{3} \frac{(1-o(1))\ln(S_G(v)/d_G(v))}{S_G(v)/d_G(v)} + \frac{1}{3} \frac{S_G(v)}{2m}\\
&\qquad \geq \frac{1}{3} \frac{(1-o(1))\ln(m^{1/3})}{d(v)} +\frac{1}{3} \frac{(1-o(1))\ln(m^{1/3})}{S_G(v)/d_G(v)} + \frac{1}{3} \frac{S_G(v)}{2m}\\
&\qquad \geq \left(\frac{(1-o(1))\ln(m^{1/3})}{d_G(v)} \cdot \frac{(1-o(1))\ln(m^{1/3})}{S_G(v)/d_G(v)} \cdot \frac{S_G(v)}{2m}\right)^{1/3},
\end{align*}
where the last line follows by the AM--GM inequality. Simplifying yields a lower bound of $(1-o(1))\left(\frac{\ln(m)^2}{18m}\right)^{1/3}$, as desired.
\end{proof}
Finally, let us give the short deduction of our last main result, namely Theorem~\ref{thm:main4}. We will need the classical Perron-Frobenius theorem on eigenvectors and eigenvalues of matrices, in the following form (see~e.g.~\cite{horn} for a reference).
\begin{theorem}\label{thm:frobenius}
Let $A\in \mathbb{R}^{n\times n}$ be a matrix with non-negative entries. Then the spectral radius $\rho(A)$ is an eigenvalue of $A$, to which there exists an eigenvector $\mathbf{u}\in \mathbb{R}^n$ with non-negative entries.
\end{theorem}

We are now ready for the deduction of Theorem~\ref{thm:main4}.
\begin{proof}[Proof of Theorem~\ref{thm:main4}]
Let $G$ be any given triangle-free graph. We will show that $\chi_f(G)\le \frac{1}{f(\rho(G)))}$, where $f$ is the function defined in Section~\ref{sec:proof}. Since $f(x)=(1-o(1))\frac{\ln(x)}{x}$, this will verify the claim of Theorem~\ref{thm:main4}.

Pause to note that $\chi_f(G)=\max\{\chi_f(G_1),\ldots,\chi_f(G_c)\}$ and similarly $\rho(G)=\max\{\rho(G_1),\ldots,\rho(G_c)\}$ holds for every graph $G$ with connected components $G_1,\ldots,G_c$. Hence, since $f$ is monotonically decreasing, it suffices to show the inequality $\chi_f(G)\le \frac{1}{f(\rho(G))}$ for all \emph{connected} triangle-free graphs. 

So let $G$ be such a graph, and let $A\in \mathbb{R}^{V(G)\times V(G)}$ be its adjacency matrix.  By definition, $A$ has non-negative entries, and hence we may apply Theorem~\ref{thm:frobenius} to find that $\rho(A)=\rho(G)$ is an eigenvalue of $A$ and that there exists a corresponding eigenvector $\mathbf{u}\in \mathbb{R}^{V(G)}$ with non-negative entries. So we have $A\mathbf{u}=\rho(G)\mathbf{u}$, which reformulated means that
$$\sum_{x\in N_G(v)}\mathbf{u}_x=\rho(G)\mathbf{u}_v$$ for every $v\in V(G)$. This equality in particular implies that if at least one neighbor of a vertex $v$ has a positive entry in $\mathbf{u}$, then so does $v$. Hence, since $G$ is a connected graph, it follows that $\mathbf{u}_v>0$ for every $v\in V(G)$. Now interpret the entries of the vector $\mathbf{u}$ as a strictly positive weight assignment to the vertices of $G$. Then, by Theorem~\ref{thm:aux}, there exists a probability distribution $\mathcal{D}$ on the independent sets of $G$ such that for every $v\in V(G)$, we have
$$\mathbb{P}_{I\sim D}[v\in I]\ge f\left(\frac{\sum_{x\in N_G(v)}\mathbf{u}_x}{\mathbf{u}_v}\right)=f(\rho(G)).$$ By definition of the fractional chromatic number, this implies that $\chi_f(G)\le \frac{1}{f(\rho(G))}$, as desired. This concludes the proof.
\end{proof}

\bibliographystyle{abbrvurl}
\bibliography{references}

\end{document}